\documentclass[a4paper,10pt]{amsart}

\usepackage{a4}
\usepackage{graphics,graphicx}
\usepackage{amssymb,amsmath}
\usepackage[all]{xy}
\usepackage{longtable}
\usepackage{multirow}

\CompileMatrices
\usepackage{hyperref}

\usepackage{tikz}
\usetikzlibrary{calc}
\usetikzlibrary{decorations.markings}
\usetikzlibrary{arrows,shapes,matrix}
\usetikzlibrary{positioning,shapes,shadows,arrows}
\usetikzlibrary{mindmap,trees,calc}
\usetikzlibrary{arrows,shapes,matrix}
\usetikzlibrary{calc}
\usetikzlibrary{arrows,shapes}
\usetikzlibrary{positioning,shapes,shadows,arrows}
\usetikzlibrary{decorations.markings}
\usetikzlibrary{decorations.pathreplacing}
\usepackage{dpfloat}
\usepackage{todonotes}

\newtheorem*{introtheorem}{Theorem}

\newtheorem{theorem}{Theorem}[section]
\newtheorem{lemma}[theorem]{Lemma}
\newtheorem{proposition}[theorem]{Proposition}
\newtheorem{corollary}[theorem]{Corollary}

\theoremstyle{definition}

\newtheorem{definition}[theorem]{Definition}
\newtheorem{example}[theorem]{Example}
\newtheorem{remark}[theorem]{Remark}

\newtheorem{construction}[theorem]{Construction}

\newtheorem{algorithm}[theorem]{Algorithm}

\newtheorem*{introalgo}{Algorithm}

\theoremstyle{remark}

\def\KK{\mathbb{K}}
\def\ZZ{\mathbb{Z}}

\def\PP{\mathbb{P}}
\def\QQ{\mathbb{Q}}

\def\OO{\mathcal O}
\def\diag{{\rm diag}}

\def\<{\langle}
\def\>{\rangle}

\def\cox{\mathcal{R}}

\newcommand{\aut}{{\rm Aut}}
\newcommand{\caut}{{\rm CAut}}

\makeatletter
\newcommand{\thickhline}{%
    \noalign {\ifnum 0=`}\fi \hrule height 1pt
    \futurelet \reserved@a \@xhline
}
\makeatother

\renewcommand{\phi}{\varphi}
\newcommand{\myhline}{\noalign{\global\arrayrulewidth.95pt}\hline
                      \noalign{\global\arrayrulewidth.2pt}}

\def\Chi{{\mathbb X}}

\def\quot{/\!\!/}
\def\mal{\! \cdot \!}

\def\rq#1{\widehat{#1}}

\def\b#1{\overline{#1}}

\def\KK{{\mathbb K}}

\def\ZZ{{\mathbb Z}}

\def\QQ{{\mathbb Q}}
\def\PP{{\mathbb P}}

\def\Cox{\cox}

\def\id{{\rm id}}

\def\Mov{{\rm Mov}}

\def\Aut{\operatorname{Aut}}
\def\Bir{\operatorname{Bir}}

\def\Cl{\operatorname{Cl}}

\def\GL{{\rm GL}}

\def\Spec{{\rm Spec}}

\def\Aut{{\rm Aut}}
\def\CAut{{\rm CAut}}

\def\spec{\Spec}

\def\orig{\Omega}
\def\Orig{\orig}

\def\KT#1{\KK[T_1,\ldots,T_{#1}]}

\def\Stab{{\rm Stab}} 
\def\Cent{{\rm Cent}} 
\def\stab{\Stab} 
\def\BBB{{\mathcal B}}

\makeatletter
\newcommand{\vast}{\bBigg@{4}}
\newcommand{\Vast}{\bBigg@{5}}
\makeatother

\usepackage{lmodern}
\makeatletter
\ifcase \@ptsize \relax% 10pt
  \newcommand{\miniscule}{\@setfontsize\miniscule{4}{5}}% \tiny: 5/6,  baselineskip should be 1.2 times the font size
\fi
\makeatother

\makeindex

\author[J.~Hausen, S.~Keicher and R.~Wolf]{J\"urgen~Hausen, Simon~Keicher and R\"udiger Wolf}

 \address{Mathematisches Institut, Universit\"at T\"ubingen,
Auf der Morgenstelle 10, 72076 T\"ubingen, Germany}
\email{juergen.hausen@uni-tuebingen.de}

\address{Departamento de Matematica\\
Facultad de Ciencias Fisicas y Matematicas\\
Universidad de Concepci{\'o}n\\ Casilla 160-C, Concepci{\'o}n, Chile}
\email{keicher@mail.mathematik.uni-tuebingen.de}
\thanks{The second author was supported by proyecto FONDECYT postdoctorado N.~3160016.}

\address{Mathematisches Institut, Universit\"at T\"ubingen,
Auf der Morgenstelle 10, 72076 T\"ubingen, Germany}
\email{}

\title[Computing automorphisms]{Computing automorphisms of Mori dream spaces}


\begin{document}

\begin{abstract}
We present an algorithm to compute the automorphism
group of a Mori dream space.
As an example calculation, we determine the 
automorphism groups of singular cubic surfaces 
with general parameters.
The strategy is to study graded automorphisms of 
an affine algebra graded by a finitely generated
abelian group and apply the results to the Cox ring.
Besides the application to Mori dream spaces, our 
results could be used for symmetry based computing,
e.g. for Gr\"obner bases or tropical varieties.
\end{abstract}

\maketitle

\section{Introduction}

We are interested in automorphism groups of
Mori dream spaces. Recall that the latter are 
normal projective varieties~$X$ with
finitely generated divisor class group
$\Cl(X)$ and finitely generated Cox ring
$$
\mathcal{R}(X)
\ = \
\bigoplus_{[D] \in \Cl(X)} \Gamma(X,\mathcal{O}(D)).
$$
The Cox ring together with an ample class
completely encodes a Mori dream space:
$X$ can be reconstructed as the GIT quotient
of $\Spec \, \mathcal{R}(X)$ associated
to the ample class in $\Cl(X)$, 
see~\cite{HuKe,ArDeHaLa}.

The basic idea is to tackle the automorphism
group of $X$ via graded automorphisms of 
its Cox ring. 
This approach was used in~\cite{Co} in
the case of complete toric varieties $X$,
where root subgroups, dimension 
and number of connected components of 
$\Aut(X)$ can be described.
In~\cite{ArHaHeLi}, the more general case
of complete rational varieties with a
torus action of complexity one was
considered,
where a description of root subgroups is 
still possible, but general satisfying 
statements on the dimension or the number 
of connected components appear to be 
difficult.
The aim is of this paper is to provide
algorithmic tools for the study automorphism
groups for arbitrary Mori dream spaces.
The main result in this regard is the
following, see Algorithm~\ref{algo:autx}:

\begin{introalgo}
\emph{Input:} the Cox ring of $X$
in terms of homogeneous generators
and relations and an ample class
in $\Cl(X)$.
\emph{Output:} a presentation of the 
Hopf algebra of $\Aut(X)$ in terms 
of generators and relations.
\end{introalgo}

This allows in particular to compute 
the dimension and the number of 
connected components of $\Aut(X)$.
We also apply the algorithm 
to detect or exclude symmetries.
In Example~\ref{ex:detectcplx1} we 
discuss the blow-up of the projective 
space in special configuration of 
six points.
As another sample computation, we 
continue the investigation of automorphism 
groups of cubic surfaces started in~\cite{Sa}
by now entering the case with parameters,
see Theorem~\ref{thm:cubic}:

\goodbreak

\begin{introtheorem}
Let $X \subseteq \PP^3$ be a
singular cubic surface
with at most ADE singularities
and parameters in its defining
equations.
Depending on the ADE singularity
type $S(X)$,
the automorphism group $\aut(X)$
for a general choice of parameters
is the following.

\goodbreak

\begingroup
\footnotesize
\begin{longtable}{cccccccc}
\myhline
$S(X)$ & $A_3$ & $A_2A_1$ & $2A_2$  & $3A_1$ & $A_2$ & $2A_1$ & $A_1$  \\
\hline %%%%%%%%%%%%%%%%%%%%%%%%%%%%%%%%%%%%%%%%%%%
$\Aut(X)$ & $\ZZ/2\ZZ$ & $\{1\}$ & $\KK^*\rtimes \ZZ/2\ZZ$ & $S_3$ &
$\{1\}$ & $\ZZ/2\ZZ$ & $\{1\}$\\
\myhline
\end{longtable}
\endgroup
\end{introtheorem}

\goodbreak

As a theoretical byproduct of our 
considerations, we observe a bound on the 
dimension of the automorphism group
of a Mori dream space, see
Proposition~\ref{prop:autbound}.
Another result is an extension statement on  
automorphisms of Mori dream spaces, 
which briefly explain here. 
Any choice of $\Cl(X)$-prime
homogeneous generators for the 
Cox ring of a given Mori dream 
space $X$ leads to an
embedding $X \subseteq Z$ into 
a toric variety $Z$ such that
the embedded $X$ inherits many 
geometric properties of the ambient $Z$;
see~\cite{ArDeHaLa}.
Our statement concerns the unit component 
$\Aut(X)^0$ of the automorphism group 
$\Aut(X)$ and the group $\Bir_2(X)$
of birational automorphisms defined up
to codimension two; see 
Theorem~\ref{thm:bir2} for the precise 
formulation.

\begin{introtheorem}
Let $X \subseteq Z$ be the toric 
embedding of a Mori dream space 
arising from the choice of pairwise
$\Cl(X)$-prime homogeneous 
generators.
Then $\Aut(X)^0$ and~$\Bir_2(X)$
are induced by the stabilizer subgroups
of $X$ in $\Aut(Z)^0$ and $\Bir_2(Z)$ 
respectively.
\end{introtheorem}

As indicated, our approach goes via
graded automorphisms of algebras
$R$ graded by finitely generated
abelian groups~$K$.
Section~\ref{sec:symalg} provides
a theoretical study of this setting
and Section~\ref{sec:algos} presents
the first algorithmic results:
its core is Algorithm~\ref{algo:quotrep},
which represents 
the graded automorphism group
explicitly
as a closed subgroup of a general linear
group.
The final Algorithm~\ref{algo:autx}
computing $\Aut(X)$ first determines
the graded automorphisms using
Algorithm~\ref{algo:quotrep}, next
computes the stabilizer of the
set of semistable points associated
with an ample class of~$X$
using Algorithm~\ref{algo:autwidehatX}
and finally arrives via an invariant
ring calculation at~$\Aut(X)$.

Note that besides our application to
Mori dream spaces, our methods can be
used to obtain  symmetries of
homogeneous ideals.
This supports symmetry based algorithms
such as the Gr\"obner basis
computations~\cite{faugere,steidel}
or the computation of tropical varieties~\cite{gfan}.
We discuss this effect in
Example~\ref{ex:groebnerfan}.
Our algorithms will be made available soon
in a suitable software package.

We want to thank Johannes Hofscheier for
his interest in our considerations and
helpful discussions.
Moreover, we are grateful to the referee
for carefully reading the text and 
for his very valuable comments.

\tableofcontents

%%%%%%%%%%%%%%%%%%%%%%%%%%%

\section{Automorphisms of graded algebras}
\label{sec:symalg}

We study algebras graded by finitely generated 
abelian groups. 
The aim is to make the graded automorphism 
group accessible for computations in the 
case of an effective, pointed grading 
as introduced below. 
The precise statements are given in 
Propositions~\ref{prop:stabpres} and
\ref{prop:autfinitepres}.

Let us fix the notation and recall
basic background.
Our ground field $\KK$ is algebraically closed 
and of characteristic zero.
A $\KK$-algebra $R$ is \emph{graded} by 
an abelian group $K$ if it comes 
with a direct sum decomposition 
into vector subspaces 
$$ 
R \ = \ \oplus_K R_w
\quad\text{such that}\quad
R_wR_{w'} \ \subseteq \ R_{w+w'}
\quad\text{for all} \quad w,w' \in K.
$$
A \emph{graded homomorphism} of 
two such algebras $R = \oplus_K R_w$ 
and $S = \oplus_L S_u$ is 
a pair $(\varphi,\psi)$, where  
$\varphi \colon R \to S$ is an 
$\KK$-algebra homomorphism and 
$\psi \colon K \to L$
a homomorphism of the grading groups 
such that $\varphi(R_w) \subseteq S_{\psi(w)}$ 
holds for all~$w \in K$.
We denote by $\aut_K(R)$ the group of graded 
automorphisms of~$R$.

If $R = \oplus_K R_w$ is integral, then 
the \emph{weight monoid} is the submonoid
$\omega(R) \subseteq K$ of all degrees 
$w \in K$ with $R_w \ne 0$.
We say that the $K$-grading of $R$ is 
\emph{effective} if~$\omega(R)$ 
generates $K$ as a group.
Moreover, we call the $K$-grading 
\emph{pointed} if $R_0= \KK$ holds 
and the convex cone in 
$K_\QQ = K \otimes_\ZZ \QQ$ 
generated by $\omega(R)$ contains 
no line.

Let $R = \oplus_K R_w$ be integral and  
the $K$-grading  pointed. 
Then we have a partial ordering on 
$\omega(R)$ defined by $w' \le w$ 
if $w = w' + w_0$ for some $w_0 \in \omega(R)$.
If $R$ is moreover finitely generated, then 
$\omega(R)$ is so and each $w \in \omega(R)$ 
admits only finitely many 
$w' \in \omega(R)$ with $w' \le w$.
With the subalgebras $R_{<w} \subseteq R$
generated by all~$R_{w'}$, where $w' < w$,
we then can figure out the unique, 
finite set of \emph{generator degrees}
$$ 
\Omega_R \ := \ \{w \in \omega(R); \ R_w \not \subseteq R_{<w} \}.
$$
Denote by $\CAut_K(R)\subseteq \Aut_K(R)$ 
the subgroup of all graded $\KK$-algebra automorphisms 
of the form $(\phi,\id)$ and consider the 
symmetry group
$$
\Aut(\Orig_R)
\ :=\ 
\left\{\psi\in \Aut(K);\ \psi(\Orig_R) = \Orig_R\right\}
\ \subseteq\ 
\Aut(K).
$$

\begin{proposition}
\label{prop:exseq}
Let $R = \oplus_K R_w$ be a finitely generated,
integral $\KK$-algebra with an effective, pointed 
grading by a finitely generated abelian group~$K$.
\begin{enumerate}
\item
Every $(\varphi,\psi) \in \Aut_K(R)$ satisfies
$\psi(\Aut(\Orig_R)) = \Orig_R$. 
In particular, there is a well-defined homomorphism 
of groups 
$$
\qquad
\pi_{\orig}\colon \Aut_K(R) \ \to \ \Aut(\Orig_R),
\qquad\qquad
(\phi,\psi) \ \mapsto \ \psi.
$$
\item
The image 
$\Gamma := \pi_{\orig}(\Aut_K(R)) \subseteq \Aut(\Orig_R)$ 
is a finite group and we have an exact sequence of affine 
algebraic groups  
$$
\qquad
\xymatrix{
1\ar[r] 
& 
\CAut_K(R)
\ar[r] 
%\ar@{}[r]|\subseteq 
& 
\Aut_K(R)\ar[r]^{\quad \pi_\orig} 
& 
\Gamma\ar[r] & 1
}.
$$
\end{enumerate}
\end{proposition}

\begin{proof}
As the set of generator degrees is unique, 
it is invariant under graded automorphisms.
For the fact that all groups of the 
sequence are affine algebraic, see for 
example~\cite{ArHaHeLi}.
The rest is obvious.
\end{proof}

By a \emph{minimal presentation}
of a $K$-graded $\KK$-algebra $R$ we mean  a 
$K$-graded polynomial ring $S := \KT{r}$ 
with $K$-homogeneous variables $T_i$
and a graded epimorphism 
$(\pi, \kappa) \colon \KT{r} \to R$
such that $\kappa \colon K \to K$ is
an isomorphism of groups and we have 
$$
\ker(\pi) 
\ \subseteq \ 
\<T_1,\ldots,T_r\>^2.
$$

We now fix the setting for our study of the group 
$\Aut_K(R)$ of graded algebra automorphisms:
$K$ will always be a finitely generated abelian 
group, $R$ an integral, finitely generated 
$\KK$-algebra and the $K$-grading of $R$ will 
always be effective and pointed.
Moreover, we fix a minimal presentation
of $(\pi,\id) \colon S \to R$, 
write $I := \ker(\pi)$ for the ideal of relations
and will frequently 
identify $R$ with~$S/I$.

\begin{remark}
For the minimal presentation 
$S \to R$ fixed above, the $K$-grading on 
$S$ is effective and pointed as 
well.
Moreover $\Orig_S = \Orig_R$ consists 
of the degrees $\deg(T_i) \in K$ of 
the variables $T_i \in S$.
\end{remark}

We now relate the graded automorphism 
group of $R$ to that of $S$.
For any group action $G \times M \to M$
and subsets $H \subseteq G$ and $N \subseteq M$,
the associated \emph{stabilizer} is 
the subset
$$
\Stab_N(H)
\ := \ 
\{g \in H; \ g \cdot N = N\}
\ \subseteq \ 
G.
$$
If the action is algebraic and 
$H \subseteq \GL(n)$ is closed (a subgroup), 
then $\Stab_N(H)$ is closed (a subgroup)
in $G$.
Here is how the stabilizers occur in our 
setting.

\begin{proposition}
\label{prop:stab}
Consider a minimally presented $K$-graded
$\KK$-algebra  $R=S/I$ as before.
Then we have a commutative diagram of 
affine algebraic groups with exact rows and columns
\[
\xymatrix@R=20pt{
&
1
&
1
&
&
\\
1
\ar[r]
&
\CAut_K(R)
\ar[r]
%\ar@{}[r]|\subseteq
\ar[u]
&
\Aut_K(R)
\ar[r]^{\pi_{\orig}}
\ar[u]
&
\Gamma
\ar[r]
&
1
\\
&
\ar[u]^{\Phi' \colon (\phi,\psi)\mapsto(\phi_\pi,\psi)}
\Stab_I(\caut_K(S))
\ar[r]
%\ar@{}[r]|\subseteq
&
\ar[u]_{\Phi \colon (\phi,\psi)\mapsto(\phi_\pi,\psi)}
\Stab_I(\aut_K(S))
&
&
\\
&
\ar[u]
\ker(\Phi')
\ar[r]
%\ar@{}[r]|\subseteq
&
\ar[u]
\ker(\Phi)
&
&
\\
&
\ar[u]
1
&
\ar[u]
1
&
&
}
\]
where the horizontal sequence is as in Proposition~\ref{prop:exseq},
the map~$\Phi'$ is the restriction of~$\Phi$
and $\phi_\pi$ is the unique homomorphism turning
\[
\xymatrix{
S
\ar[rr]^{(\phi,\psi)}
\ar[d]_{(\pi,\id)}
&&
S
\ar[d]^{(\pi,\id)}
\\
R
\ar[rr]_{(\phi_\pi,\psi)}
&&
R
}
\]
into a commutative diagram of graded homomorphisms.
Moreover, the following statements hold.
\begin{enumerate}
\item 
We have $\ker(\Phi) = \ker(\Phi')$.
Setting $q_i := \deg(T_i)$ and
$G:=\Stab_I(\Aut_K(S))$, we obtain 
$\ker(\Phi)$ as the subgroup
\begin{align*}
\qquad
\ker(\Phi)\, &=\, 
\left\<(\phi,\id)\in G;\ \phi(T_i)-T_i\in I_{q_i}
\text{ for all $i$}\right\>\,\subseteq\,G.
\end{align*}
\item 
If the homogeneous components
$I_{q_1},\ldots,I_{q_r}$ of the degrees 
$q_i := \deg(T_i)$ are all trivial,
then $\Phi$ and $\Phi'$ are isomorphisms.
\end{enumerate}
\end{proposition}

In the proof we need the following two lemmas,
where the second one is a certain uniqueness statement 
on minimal presentations and will also be used later.

\begin{lemma}
\label{lem:lift}
Let $(\phi,\psi) \colon S \to R$
and $(\pi,\kappa) \colon R' \to R$
be homomorphisms of $K$-graded $\KK$-algebras 
such that $S = \KT{r}$ holds, $\pi$ is 
surjective and $\kappa$ an isomorphism.
Then there is a commutative diagram
of homomorphisms of $K$-graded 
algebras:
\[
\xymatrix{
S
\ar[rr]^{(\widehat\phi,\widehat\psi)}
\ar[dr]_{(\phi,\psi)}
&&
R'
\ar[dl]^{(\pi,\kappa)}
\\
&
R
&
}
\]
\end{lemma}

\begin{proof}
Denote by $w_i \in K$ the degree of the variable
$T_i \in S$.
Each $\phi(T_i) \in R$ has a preimage 
$f_i' \in R'_{\kappa^{-1}(\psi(w_i))}$ under 
$(\pi,\kappa)$.
Define $\widehat \phi$ by $\widehat \phi(T_i) := f_i'$
and set $\widehat \psi :=  \kappa^{-1} \circ \psi$.
\end{proof}

\begin{lemma}
\label{lem:minpres}
Consider two minimal presentations 
$(\pi_i,\kappa_i) \colon S\to R$ 
of a $\KK$-algebra $R$ with an 
effective, pointed $K$-grading.
Then there is an isomorphism $(\phi,\psi)$ 
of $K$-graded algebras fitting into the 
commutative diagram
\[
\xymatrix{
S
\ar[rr]^{(\phi,\psi)}
\ar[dr]_{(\pi_1,\kappa_1)}
&&
S
\ar[dl]^{(\pi_2,\kappa_2)}
\\
&
R
&
}
\]
\end{lemma}

\begin{proof}
According to Lemma~\ref{lem:lift}, there are
graded homomorphisms $(\phi_i,\psi_i) \colon S \to S$
such that the following diagram is commutative
\[
\xymatrix{
S
\ar@<2pt>[rr]^{(\phi_1,\psi_1)}
\ar[dr]_{(\pi_1,\kappa_1)}
&&
S
\ar@<2pt>[ll]^{(\phi_2,\psi_2)}
\ar[dl]^{(\pi_2,\kappa_2)}
\\
&
R
&
}
\]
In particular, $\tau := \phi_2\circ \phi_1 \colon S \to S$ 
is a degree preserving graded algebra endomorphism with 
$\pi_1\circ\tau = \pi_1$.
Using minimality of the presentation $(\pi_1, \kappa_1) \colon S \to R$, 
we obtain
$$
T_i - \tau(T_i) 
\ \in \ \ker(\pi_1)
\  \subseteq \ \<T_1,\ldots,T_r\>^2.
$$ 

We show that $\tau$ is surjective. 
For this, it suffices to show $\tau(S_w) = S_w$ 
for all $w \in \Orig_S$.
Since $\tau$ preserves degrees, we have 
$\tau(S_w) \subseteq S_w$.
To verify equality, we proceed by induction 
on the maximal length~$k(w)$ of the possible 
decompositions
$$
w \ = \ w_1 + \ldots + w_k,
\qquad
w_i \ \in \ \Orig_S
\quad
\text{indecomposable}.
$$ 
If $k(w) = 1$, then $w$ is indecomposable in $\Orig_S$
and $S_w$ is generated as a vector space by some of 
the $T_i$. Each $T_i \in S_w$ is fixed by $\tau$
which gives $\tau(S_w) = S_w$.
Assume $k(w) > 1$. Then we find a basis 
of the vector space $S_w$ of the form 
$$ 
(T_{i_1}, \ldots, T_{i_l}, h_1,\ldots h_m),
\qquad\qquad
h_j \ \in \ \<T_1,\ldots,T_r\>^2.
$$
Each $h_j$ is a polynomial in variables $T_i$ 
of degree $w_i \in \Orig_S$ such that $w_i < w$.
The latter implies $k(w_i) < k(w)$.
By induction hypothesis, these $S_{w_i}$, 
and hence $h_j$, are in the image of $\tau$.
This proves surjectivity of $\tau$.

We now immediately obtain bijectivity of $\tau$,
because each restriction $\tau \colon S_w  \to S_w$ 
is a surjective linear endomorphism of a 
finite-dimensional vector space and hence bijective.
Thus, $\tau = \phi_2\circ \phi_1$ is an isomorphism.
Then the graded homomorphisms~$\phi_i$ must 
be isomorphisms as well and we see that 
$(\phi,\psi) := (\phi_1,\psi_1)$ is as wanted.
\end{proof}

\begin{proof}[Proof of Proposition~\ref{prop:stab}]
We show that $\Phi$ is surjective.
Given a graded automorphism, $(\phi_0, \psi_0) \in \Aut_K(R)$,
we have two minimal presentations
$$
(\pi,\id) \colon S \ \to \ R, 
\qquad\qquad
(\phi_0, \psi_0) \circ (\pi,\id) \colon S \ \to \ R.
$$ 
Lemma~\ref{lem:minpres} provides us with an
isomorphism $(\phi,\psi)$ of graded $\KK$-algebras
fitting into the commutative diagram
\[
\xymatrix{
S
\ar[rr]^{(\phi,\psi)}
\ar[d]_{(\pi,\id)}
&&
S
\ar[d]^{(\pi,\id)}
\\
R
\ar[rr]_{(\phi_0,\psi_0)}
&&
R
}
\]
Clearly $\psi = \psi_0$ and thus $(\phi,\psi)$ is 
the desired preimage of $(\phi_0,\psi_0)$ under~$\Phi$. 
This establishes the first part.
For the statements on the kernel, one directly
verifies 
\begin{eqnarray*}
\ker(\Phi)
&=&
\<(\phi,\id)\in \Aut_K(S); \ \phi_\pi(\b T_i) = \b{T_i}\in R\text{ for all } i \>\\
&=&
\<(\phi,\id)\in \Aut_K(S); \ \phi(T_i) - T_i \in I\text{ for all } i \>
\\
&=&
\<(\phi,\id)\in \Aut_K(S);\ \phi(T_i) - T_i \in I_{q_i}\text{ for all } i\>.
\end{eqnarray*}
where $\varphi_\pi \colon R \to R$ is the unique homomorphism 
induced by $(\varphi,\psi)$ as introduced in the assertion.
\end{proof}

Similar to the unique set of generator 
degrees $\Omega_R \subseteq K$ of a
$K$-graded $\KK$-algebra~$R$,
we have the unique set $\Omega_I \subseteq K$
of \emph{ideal generator degrees} 
for the homogeneous ideal $I \subseteq S$: 
let $I_{<w} \subseteq S$ be the ideal 
generated by all ideal components 
$I_{w'} \subseteq I$ with $w' < w$ and set
$$ 
\Omega_I 
\ := \ 
\{w \in K; \; I_w \not\subseteq I_{<w} \}.
$$

\begin{proposition}
\label{prop:idweights}
Every automorphism
$(\varphi,\psi) \in \Aut_K(R)$
satisfies $\psi(\Omega_I) = \Omega_I$.
\end{proposition}

\begin{proof}
According to Lemma~\ref{lem:minpres},
the graded automorphism $(\varphi,\psi)$ 
admits a lifting $(\widehat \varphi,\psi)$
with respect to the minimal presentation 
$(\pi,\id) \colon S \to R$.
\end{proof}

\begin{construction}
\label{constr:VandW}
Consider $S = \KT{r}$ and $G := \Aut_K(S)$.
Then we have finite dimensional 
vector subspaces
$$ 
V 
\ := \ 
\bigoplus_{w \in \Omega_S} S_w
\ \subseteq \ 
S,
\qquad\qquad
W
\ := \ 
\bigoplus_{u \in \Omega_I} S_u
\ \subseteq \ 
S.
$$
Moreover, $V$ is invariant under 
the (linear) $G$-action on $S$
and the induced representation 
$G \to \GL(V)$ is faithful. 
\end{construction}

\begin{proof}
Since the grading is pointed and 
$\Omega_S$ as well as $\Omega_I$ 
are finite, $V$ and $W$ are of finite
dimension.
Proposition~\ref{prop:exseq}~(i) 
guarantees that $V$ is $G$-invariant
and the induced representation is 
faithful, because $V$ generates $S$ 
as a $\KK$-algebra.
\end{proof}

The following observation will allow us to 
compute the stabilizer $\Stab_I(G)$,
which is an essential step in the computation
of $\Aut_K(R)$:

\begin{proposition}
\label{prop:stabpres}
Notation as in Construction~\ref{constr:VandW}.
Set $I_W := I \cap W$. Then $I_W$ generates 
the ideal~$I$ and the stabilizer $\Stab_I(G)$ 
is given as 
$$
\Stab_I(G)
\ = \ 
\Stab_{I_W}(G)
\ = \ 
\{g \in G; \ g \mal I_W = I_W\}
\ \subseteq \
G.
$$
\end{proposition}

Finally, we arrive at a finite dimensional 
faithful representation of $\Aut_K(R)$ used 
in our computations.

\begin{proposition}
\label{prop:autfinitepres}
Notation as in Construction~\ref{constr:VandW}.
Set $I_V := I \cap V$. 
Then the minimal presentation 
$(\pi,\id) \colon S \to R$ 
induces an isomorphism of vector 
spaces
$$ 
V / I_V \ \to \ \bigoplus_{w \in \Omega_R} R_w.
$$
Moreover, $I_V$ is invariant under $\Stab_I(G)$, 
we have an induced representation 
$\varrho \colon \Stab_I(G) \to \GL(V/I_V)$
and an isomorphism
$$
\varrho(\Stab_I(G)) \ \cong \ \Aut_K(R).
$$
\end{proposition}

\begin{proof}
Proposition~\ref{prop:stab} gives us the desired isomorphism.
\end{proof}

\begin{corollary}
\label{cor:autdimbound}
Let $K$ be a finitely generated 
abelian group and $R = \oplus_{w \in K} R_w$ 
an integral, finitely generated $\KK$-algebra
with an effective, pointed $K$-grading.
Then the dimension of the group of graded automorphisms 
is bounded by
$$ 
\dim(\Aut_K(R))
\ \le \ 
\sum_{w \in \Omega_R} \dim(R_w)^2.
$$ 
\end{corollary}

%%%%%%%%%%%%%%%%%%%%%%%
\section{Basic Algorithms}
\label{sec:algos}

We present the algorithms for computing 
the automorphism group $\Aut_K(R)$ of a 
$\KK$-algebra~$R$ graded by a finitely 
generated abelian group $K$.
Our algorithms are formulated in a 
manner allowing, up to standard 
considerations in linear algebra, 
a direct implementation.

We work in the setting of 
Section~\ref{sec:symalg}.
In particular, the $K$-grading of $R$ 
is effective and pointed.
Moreover, we fix a minimal presentation 
$(\pi,\id)\colon S \to R$
with a polynomial ring $S:=\KT{r}$ 
and denote $I := \ker(\pi)$;
such a minimal presentation is directly 
obtained by removing successively 
redundant generators from any presentation
by generators and relations.

Recall that we have $\Orig_S = \Orig_R$ for the respective 
sets of generator weights of $S$ and $R$.  
A first basic step is to determine the automorphism group 
$\Aut(\Orig_S)$.

\begin{remark}[Computing $\Aut(\Orig_S)$]
\label{rem:autorig} 
Represent the grading group $K$ as a direct sum of 
a free part and its $p$-torsion parts, i.e.,
$$
K
\ = \ 
\ZZ^k \oplus \bigoplus_{i=1}^l K_{p_i},
\qquad\qquad
K_{p_i}
\ = \
\bigoplus_{j=1}^{n_i}
\ZZ/p_i^{k_{ij}}\ZZ
$$
with pairwise different primes $p_i$ and $k_{ij} \le k_{ij+1}$.
Set $n := k + n_1 + \ldots + n_l$. 
Then the automorphisms $\psi \colon K \to K$
are induced by integral block matrices:
$$
\hspace*{3cm}
\vcenter{ 
\xymatrix{
\ZZ^n 
\ar[r]^{A}
\ar[d]
&
\ZZ^n
\ar[d]
\\
K
\ar[r]_\psi
&
K
}
},
\qquad\qquad
\vcenter{
$
A
\ = \
\left[
\begin{array}{cc}
 B & 0
\\
 C & D
 \end{array}
\right]
$
}
$$
with $B \in \GL(k,\ZZ)$ and $D$ suitably
block diagonal, see~\cite{TuVo,autfinitegrps}
for the precise conditions on $D$.
The entries of $D$ and $C$ are 
nonnegative and bounded by the 
smallest annihilator $a > 0$ of 
the torsion part of $K$. 
The subgroup $\Aut(\Orig_S) \subseteq \Aut(K)$ 
is now obtained as follows.
\begin{itemize}
\item 
Determine the (finitely many)  $B \in \GL(k,\ZZ)$
stabilizing the set of the $\ZZ^k$-parts of 
the generator degrees~$\Omega_S$.
\item 
Determine the (finitely many) matrices~$A$ with a
$B$-block from the previous step and figure out 
those defining an automorphism of $K$ 
stabilizing~$\Omega_S$.
\end{itemize}
\end{remark}

As a second step, we make Construction~\ref{constr:VandW}
explicit and realize $\Aut_K(S)$ via concrete 
equations as a subgroup of a general linear group
$\GL(n)$.
For later applications, we include the treatment 
of subgroups of the following type in $\Aut_K(S)$: 
given a subgroup
$\Sigma \subseteq \Aut(\Orig_S)$, set
$$ 
\Aut_{K,\Sigma}(S)
\ := \ 
\{(\varphi,\psi) \in \Aut_K(S); \; \psi \in \Sigma\}
\ \subseteq \ 
\Aut_K(S).
$$

\begin{construction}
\label{con:repres}
Consider the $K$-graded ring $S = \KT{r}$ and
write $\Orig_S = \{w_1, \ldots, w_s\}$ with
pairwise different $w_j$.
For each $i= 1, \ldots, s$, set $d_i := \dim_{\KK}(S_{w_i})$ 
and fix a basis~$\mathcal{B}_i$ consisting of 
monomials for the homogeneous component $S_{w_i}$.
The concatenation 
$$
\mathcal{B} 
\ := \ 
(\mathcal{B}_1, \ldots, \mathcal{B}_s)
$$ 
is a basis for the direct sum $V := \oplus_i S_{w_i}$
of all $S_{w_i}$.
Now, every graded automorphism $(\phi,\psi)\in\Aut_K(S)$ 
defines a linear automorphism
$$
V
%\bigoplus_{i=1}^{s} S_{w_i} 
\to \ 
V,
%\bigoplus_{i=1}^{s} S_{w_i},
\qquad 
S_{w_i} \ni f 
\ \mapsto \ 
\varphi(f) \in S_{\psi(w_i)}.
$$
We denote by $A(\phi,\psi) \in \GL(n)$,
where $n = d_1 + \ldots + d_s$, 
the matrix representing this map with respect 
to $\mathcal{B}$. 
This gives rise to a faithful matrix representation
$$ 
\varrho \colon
\Aut_K(S) \ \to \ \GL(n), 
\qquad 
(\phi,\psi) \ \mapsto \ A(\phi,\psi).
$$
\end{construction}

% \begin{construction}
% \label{constr:matonpol}
% Notation as in Construction~\ref{con:repres}.
% Observe that every variable $T_i \in S$ 
% occurs as a basis vector $b_{j(i)}$ in the 
% basis $\mathcal{B}$.
% For $A \in \GL(n)$ and $f \in S$ set
% $$ 
% A \cdot f
% \ := \ 
% f(A^t b_{j(1)}, \ldots, A^t b_{j(r)})
% \ \in \ 
% S.
% $$
% \end{construction}

\begin{construction}
\label{constr:matonpol}
Notation as in Construction~\ref{con:repres}.
Via the basis $\mathcal{B}$, 
every matrix $A \in \GL(n)$ defines a linear 
endomorphism $\varphi_A \colon V \to V$.
Since the variables $T_1, \ldots, T_r$ of 
$S$ belong to $V$ we can define for every 
$f \in S$ the polynomial
$$ 
A \cdot f
\ := \ 
f(\varphi_A(T_1), \ldots, \varphi_A(T_r))
\ \in \ 
S.
$$
\end{construction}

\begin{definition}
\label{def:autcond}
Notation as in Constructions~\ref{con:repres}
and~\ref{constr:matonpol}.
Let $\Sigma \subseteq \Aut(\Orig_S)$
be a subgroup.
We introduce three types of matrices:
\begin{enumerate}
\item
a matrix $A \in \GL(n)$ is \emph{$S$-admissible} 
if for any two monomials $T^{\nu_1}$, $T^{\nu_2}$ 
such that the degrees of $T^{\nu_1}$, $T^{\nu_2}$ 
and $T^{\nu_1+\nu_2}$ belong to 
$\Omega_S = \{w_1,\ldots,w_s\}$ we have 
$$
A\cdot (T^{\nu_1}T^{\nu_2}) \ = \ (A\cdot T^{\nu_1})(A\cdot T^{\nu_2}). 
$$
\item
a matrix $A \in \GL(n)$ is \emph{$\Orig_S$-diagonal} 
if it is block diagonal 
with blocks $A_1,\ldots,A_s$, where $A_i \in \GL(d_i)$,
\item
a matrix $B \in \GL(n)$ is 
\emph{$\Sigma$-permuting}
if there is a $\sigma \in \Sigma$ such 
that~$B$ 
sends each $\mathcal{B}_i$ bijectively 
to $\mathcal{B}_j$, where $w_j = \sigma(w_i)$.
\end{enumerate}
\end{definition}

\begin{proposition}
\label{prop:autcond}
Notation as in Constructions~\ref{con:repres}
and~\ref{constr:matonpol}.
Then $\varrho(\Aut_{K,\Sigma}(S))$ consists exactly of
the $S$-admissible matrices $AB \in \GL(n)$, 
where $A \in \GL(n)$ is $\Orig_S$-diagonal and
$B \in \GL(n)$ is $\Sigma$-permuting. 
\end{proposition}

\begin{proof}
First let $(\phi,\psi) \in \Aut_{K,\Sigma}(S)$.
Then, for each $i = 1, \ldots, s$, the linear 
isomorphism $\phi \colon V \to V$ restricts 
to a linear isomorphism 
$\phi_i \colon V_{w_i} \to V_{\psi(w_i)}$. 
Denote by $A_i \in \GL(d_i)$ the representing 
matrix of $\phi_i$ with respect to the bases 
$\mathcal{B}_i$ and~$\mathcal{B}_{j(i)}$, where 
$j(i)$ is defined via $\psi(w_i) = w_{j(i)}$. 
Moreover, let $B \in \GL(n)$ denote the 
permutation matrix sending the 
basis vectors $v_{ik}$ of 
$\mathcal{B}_i = (v_{i1}, \ldots, v_{id_i})$
to the basis vectors $v_{j(i)k}$ of 
$\mathcal{B}_{j(i)} = (v_{j(i)1}, \ldots, v_{j(i)d_i})$.
Then $A := \diag(A_{j(1)},\ldots,A_{j(s)})$ is 
$\Omega_S$-diagonal,
$B$ is $\Sigma$-permuting and we have 
$A(\phi,\psi) = AB$.
As the matrix $AB$ represents the restriction of the 
algebra homomorphism $\phi$ to the vector subspace 
$V \subseteq S$, it is compatible with 
the multiplication of polynomials and thus 
$S$-admissible.

Conversely, let $A \in \GL(n)$ be $\Omega_S$-diagonal 
and $B \in \GL(n)$ be $\Sigma$-permuting with respect
to $\psi \in \Sigma$ such that
$AB$ is $S$-admissible.
Via the basis $\mathcal{B}$, the matrix
$AB$ defines the linear automorphism 
$\varphi_{AB} \colon V \to V$.
Moreover, $f \mapsto AB \cdot f$ defines 
a graded algebra automorphism 
$(\varphi,\psi) \colon S \to S$, where 
$\psi \in \Sigma$ belongs to $B$ as fixed 
above.
Since $AB$ is $S$-admissible, we see that
$\varphi$ and $\varphi_{AB}$ coincide on the 
monomials lying in $V$.
Being linear maps, $\varphi$ and $\varphi_{AB}$
then coincide on $V$ and we conclude 
$AB = A(\varphi,\psi)$.
\end{proof}

\begin{corollary}
\label{cor:autcond}
In the special cases $\Sigma = \{\id_K\}$ 
and $\Sigma = \Aut(\Orig_S)$,
Proposition~\ref{prop:autcond} gives the following.
\begin{enumerate}
\item 
The image $\varrho(\CAut_K(S))$ consists exactly 
of the matrices $A \in \GL(n)$ which are 
$\Orig_S$-diagonal and $S$-admissible.
\item
The image $\varrho(\Aut_K(S))$ consists exactly of
the $S$-admissible matrices $AB \in \GL(n)$, 
where
$A \in \GL(n)$ is $\Orig_S$-diagonal and
$B \in \GL(n)$ is $\Aut(\Orig_S)$-permuting. 
\end{enumerate}
\end{corollary}

In the subsequent algorithms, we say that an
ideal $\mathfrak{a} \subseteq \KK[T_1, \ldots, T_n]$ 
\emph{describes} a subset $X \subseteq \KK^n$ if
$X$ equals the zero set of $\mathfrak{a}$
and we say that $\mathfrak{a}$ \emph{defines}
a closed subset $X \subseteq \KK^n$ if $\mathfrak{a}$ 
equals the vanishing ideal of $X$.

\begin{algorithm}[Computing~$\Aut_{K,\Sigma}(S)$]
\label{algo:autks}
\emph{Input:}
the $K$-graded polynomial ring $S$
and a subgroup $\Sigma \subseteq \Aut(\Orig_S)$.
\begin{itemize}
\item 
Compute a basis $\BBB =(b_1,\ldots,b_n)$ for $V = \oplus_i S_{w_i}$
as in Construction~\ref{con:repres}.
\item 
Set $S' := \KK[T_{ij}; \; 1 \le i,j \le n]$.
\item
Let $I' \subseteq S'$ be an ideal describing the 
$S$-admissible matrices in~$\GL(n)$. 
\item
Let $J \subseteq S'$ be an ideal describing the
$\Orig_S$-diagonal matrices in~$\GL(n)$. 
\item
Compute the (finite) set $\mathfrak{B}$ of 
$\Sigma$-permuting matrices.
\item
For each $B \in \mathfrak{B}$ form the product $J := J \cdot (B^*J)$.
\end{itemize}
\emph{Output:} 
the ideal 
$I'+J \subseteq \KK[T_{ij}; \; 1 \le i,j \le n]$.
It describes the subgroup $\Aut_{K,\Sigma}(S) \subseteq \GL(n)$.
\end{algorithm}

We proceed with computing equations for 
the stabilizer of the ideal of relations 
$I \subseteq S$ of $R$ 
in $\Aut_{K,\Sigma}(S) \subseteq \GL(n)$,
where we follow the lines of 
Proposition~\ref{prop:stabpres}.

\begin{algorithm}[Computing $\Stab_I(\Aut_{K,\Sigma}(S))$]
\label{algo:stab}
\emph{Input:}
the $K$-graded polynomial ring $S$ and the defining 
ideal $I \subseteq S$ of $R$
and a subgroup $\Sigma \subseteq \Aut(\Orig_S)$.
\begin{itemize}
\item
Let $I'+J \subseteq S' := \KK[T_{ij}; \; 1 \le i,j \le n]$
be the output of Algorithm~\ref{algo:autks}. 
\item 
Determine $\Omega_I$ and the vector space $W := \oplus_{\Omega_I} S_u$.
\item
Determine a basis $(h_1,\ldots,h_l)$ for 
the vector subspace $I_W = I \cap W \subseteq W$.
\item
Compute linear forms $\ell_{1},\ldots,\ell_{m} \in W^*$ 
with $I_W$ as common zero set.
\item 
With $T=(T_{ij})$ and the $\GL(n)$-action on $S$
from Construction~\ref{constr:matonpol} 
define the ideal 
$$
J'
\ := \
\left\< 
\ell_{i}\left(T\cdot h_j\right); \ 1\leq i\leq m ,\ 1\leq j\leq l
\right\>
\ \subseteq\ S'.
$$
\end{itemize}
\emph{Output:} 
the ideal
$I'+J+J' \subseteq S'$. 
It describes the subgroup $\stab_I(\Aut_{K,\Sigma}(S)) \subseteq \GL(n)$.
\end{algorithm}

Finally, we implement Proposition~\ref{prop:autfinitepres}
to compute in particular $\Aut_K(R)$ as a subgroup 
of a suitable general linear group $\GL(k)$.
We still consider subgroups 
$\Sigma \subseteq \Aut(\Orig_S) = \Aut(\Orig_R)$ and set
$$ 
\Aut_{K,\Sigma}(R)
\ := \ 
\{(\varphi,\psi) \in \Aut_K(R); \; \psi \in \Sigma\}
\ \subseteq \ 
\Aut_{K}(R).
$$

\begin{algorithm}[Computing $\Aut_{K,\Sigma}(R)$]
\label{algo:quotrep}
\emph{Input:}
the $K$-graded polynomial ring $S$ and the defining 
ideal $I \subseteq S$ of $R$
and a subgroup $\Sigma \subseteq \Aut(\Orig_S)$.
\begin{itemize}
\item
Let $I'+ J + J' \subseteq S' := \KK[T_{ij}; \; 1 \le i,j \le n]$
be the output of Algorithm~\ref{algo:stab}. 
\item
Determine the vector subspace $I_V = I \cap V$ of $V = \oplus_{\Orig_S} S_w$.  
\item 
Determine a basis $u_1, \ldots, u_k$ for $V/I_V$ and set 
$S'' := \KK[T_{ij}; \ 1 \le i,j \le k]$.
\item
Compute a defining ideal $J'' \subseteq S''$ for the image of the 
homomorphism 
$$
\GL(n) \supseteq \Stab_I(\Aut_{K,\Sigma}(S)) 
\ \to \ 
\GL(k) = \Aut(V/I_V)
$$
of algebraic groups arising from 
the projection $\KK^n = I_V \to V/I_V = \KK^k$. 
\end{itemize}
\emph{Output:}
the ideal $J''\subseteq \KK[T_{ij};\,1\leq i,j\leq k]$.
It defines 
the subgroup $\Aut_{K,\Sigma}(R) \subseteq \GL(k)$.
\end{algorithm}

\begin{remark}
Algorithm~\ref{algo:stab} makes no use of Gr\"obner bases,
whereas the image computation in the last step of 
Algorithm~\ref{algo:quotrep} involves determining
the kernel of a ring map, which usually is performed
via Gr\"obner basis computations.
However, if all homogeneous components $I_{w} \subseteq I$ 
with $w\in \Omega_S$ are trivial, 
then Proposition~\ref{prop:stab} provides a 
describing ideal for 
$\Aut_K(R) \cong \Stab_I(\Aut_K(S))$ 
without any Gr\"obner basis computation.
\end{remark}

\begin{remark}
The output of Algorithm~\ref{algo:quotrep}
allows to compute the dimension 
as well as the number of connected 
components of~$\Aut_K(R)$ by standard
algorithms using Gr\"obner bases.
\end{remark}

\begin{remark}[Computing $\Gamma$
from Proposition~\ref{prop:exseq}]
\label{rem:cautgamma}
Compute the groups
$\CAut_K(R)$ and $\Aut_K(R)$
with Algorithm~\ref{algo:autks}.
Then the factor group 
$\Gamma$ consists of those $\sigma \in \Aut(\Omega_S)$
that admit an $\Aut(\Omega_S)$-permuting matrix $B$ 
with 
$$
B^* \Stab_I(\Aut_K(S)) \cap \Stab_I(\Aut_K(S))
\ \ne \ 
\emptyset.
$$ 
The $\Aut(\Omega_S)$-permuting matrices are computed in 
Algorithm~\ref{algo:autks}
and the displayed condition is computable in 
terms of the describing ideal of the stabilizer,
which in turn is provided by Algorithm~\ref{algo:stab}.
\end{remark}

\begin{remark} 
Proceeding similarly as in Algorithm~\ref{algo:stab}, 
one can compute the \emph{transporter} of two homogeneous 
ideals $I_1,I_2\subseteq S$, that is,
the closed subset 
$\{g \in \Aut_K(S); \ g\cdot I_1\subseteq I_2\}$. 
\end{remark}

\begin{example}[$A_32A_1$-singular Gorenstein log del Pezzo $\KK^*$-surface]
\label{ex:running1}
Recall, for example from~\cite{ArDeHaLa, Hu:diss}, 
that the Cox ring $R$ of the singular Gorenstein 
log del Pezzo $\KK^*$-surface $X$ with singularity type $A_32A_1$ is
$$
R\, =\, S/I,\qquad S\, :=\, \KT{5},\qquad I\,:=\,\<T_1T_2 + T_3^2 + T_4^2\>,
$$
where $R$ is effectively and pointedly graded by $K:=\ZZ^2\oplus \ZZ/2\ZZ$ 
with generator degrees $w_i = \deg(T_i)$ given by
\begin{center}
\begin{minipage}{5.5cm}
\vspace*{-.1cm}
\begin{align*}
[w_1,\ldots,w_5]
\,&:=\,
\left[
\mbox{\footnotesize $
\begin{array}{rrrrr}
1 & 1 & 1 & 1 & 1 \\
1 & -1 & 0 & 0 & 1\\
\b 1 & \b 1 & \b 1 & \b 0 & \b 0
\end{array}
$}
\right],
\end{align*}
\end{minipage}
\qquad \qquad 
\begin{minipage}{3.5cm}
\footnotesize
\begin{tikzpicture}[scale=.7]

\fill[color=black!40] (0,0) -- (1,1) -- (1,-1) -- cycle;
\draw[->, dashed] (-.5,0) -- (1.4,0);
\draw[->, dashed] (0,-1.13) -- (0,1.13);

\draw[thick] (0,0) -- (1,1);
\draw[thick] (0,0) -- (1,-1);

\fill[color=black] (1,1) circle (1.6pt) node[anchor=west]{$w_1, w_5$};
\fill[color=black] (1,0) circle (1.6pt) node[anchor=north west]{$w_3, w_4$};
\fill[color=black] (1,-1) circle (1.6pt) node[anchor=west]{$w_2$};

\draw (3,0) node[anchor=west]{$\subseteq\,K\otimes \QQ$};
\end{tikzpicture}
\end{minipage}
\end{center}
Observe that we have $\Orig_S=\{w_1,\ldots,w_5\}$ and 
Remark~\ref{rem:autorig} yields
the symmetries of generator weights
$$
\Aut(\Orig_S)
\ \cong\  
\left\{
\id,\,
\psi_1
\right\}
\ \cong\  
\ZZ/2\ZZ
,\qquad
\psi_1\colon K\to K,\quad
g\mapsto 
\left[\mbox{\tiny
$\begin{array}{rrr}
1 & 0 & 0 \\ 
0 & 1 & 0 \\ 
1 & 1 & 1
\end{array}$}
\right]\cdot g.
$$
Algorithm~\ref{algo:autks} first produces 
the bases $\BBB_i=(T_i)$ for the components 
$S_{w_i}$ from Construction~\ref{con:repres}
and then returns the description of $G:=\Aut_K(S)$:
\begin{eqnarray*}
 G
 &\cong&
 \hphantom{\cup}
 \left\{
\diag(a_{1,1},a_{2,2},a_{3,3}, a_{4,4}, a_{5,5});\ 
a_{i,j}\in \KK^*
 \right\}
 \  \cup
 \\
 &&
 \left\{
\left[\mbox{\tiny $ \begin{array}{rrrrr}
a_{1,1} & 0 & 0 & 0 & 0\\
0 & a_{2,2} & 0 & 0 & 0\\
0 & 0 & 0 & a_{3,4} & 0\\
0 & 0 & a_{4,3} & 0 & 0\\
0 & 0 & 0 & 0 & a_{5,5}\\
\end{array}
$}\right];\  
a_{i,j}\in \KK^*
 \right\}
 \ \subseteq\ \GL(5).
\end{eqnarray*}
Note that in this case, 
we did not need equations for the $S$-admissible property.
These matrices correspond to explicit automorphisms 
as noted in Construction~\ref{con:repres}.
For instance, the second type of matrices in the above description of $G$
 belongs to elements $(\phi,\psi)\in\Aut_K(S)$ with $\psi =\psi_1$
 and $\phi$ given by
 \begin{gather*}
T_1\mapsto a_{1,1}T_1,\qquad
T_2\mapsto a_{2,2}T_2,\qquad
T_3\mapsto a_{3,4}T_4,\qquad
T_4\mapsto a_{4,3}T_3,\qquad
T_5\mapsto a_{5,5}T_5.
\end{gather*}
% \begin{gather*}
% T_1\mapsto T_1,\qquad
% T_2\mapsto T_2,\qquad
% T_3\mapsto a_{3,3}T_4,\qquad
% T_4\mapsto a_{4,4}T_3,\qquad
% T_5\mapsto T_5.
% \end{gather*}
Since all $I_{\deg(T_i)}$ are trivial, 
we can use Algorithm~\ref{algo:stab} 
and see that $\Aut_K(R)$, as a subgroup 
of $\GL(5)$, is given as a union of two sets,
each of which consists of two connected 
components: 
\begin{align*}
\Aut_K(R)
\ = \ 
&\left\{
\diag(a_{1,1},a_{2,2},a_{3,3}, a_{4,4}, a_{5,5})
\in 
\GL(5) 
;\ 
\mbox{\footnotesize $
\begin{array}{l}
a_{3,3}^2=a_{4,4}^2,\\ 
a_{1,1}a_{2,2}=a_{3,3}^2
\end{array}
$}
\right\}
\\
\cup\ 
&\left\{
\left[\mbox{\tiny $ \begin{array}{rrrrr}
a_{1,1} & 0 & 0 & 0 & 0\\
0 & a_{2,2} & 0 & 0 & 0\\
0 & 0 & 0 & a_{3,4} & 0\\
0 & 0 & a_{4,3} & 0 & 0\\
0 & 0 & 0 & 0 & a_{5,5}\\
\end{array}
$}\right]\in \GL(5);\ 
\mbox{\footnotesize $
\begin{array}{l}
a_{3,4}^2=a_{4,3}^2,\\
a_{1,1}a_{2,2}=a_{4,3}^2
\end{array}
$}
\right\}.
\end{align*}
In particular, we obtain $\dim(\Aut_K(R)) = 3$
and we see that the unit component $\Aut_K(R)^0$
is of index four in $\Aut_K(R)$.
Moreover, we directly verify that $\Aut_K(R)$ 
is isomorphic to the semidirect product 
% $(\ZZ/2\ZZ \times \ZZ/2\ZZ) \ltimes (\KK^*)^3$.
$\ZZ/2\ZZ \ltimes (\ZZ/2\ZZ \times (\KK^*)^3)$.
\end{example}

%%%%%%%%%%%%%%%%%%%%%%%%%%%%%%
\section{Automorphism groups of Mori dream spaces}
\label{sec:sym-mds}

Here we apply the results of the preceding sections 
to study automorphism groups of Mori dream spaces.
Recall that a normal projective variety $X$ is a 
\emph{Mori dream space} if it has finitely generated 
divisor class group $\Cl(X)$ and finitely generated 
\emph{Cox ring}
$$ 
\mathcal{R}(X) 
\ = \ 
\bigoplus_{[D] \in \Cl(X)} \Gamma(X, \mathcal{O}(D)),
$$
where the case of torsion in $\Cl(X)$ requires
a little care in this definition; 
see~\cite[Sec.~1.4.2]{ArDeHaLa}
for details.
Any Mori dream space $X$ has a  
\emph{total coordinate space} 
$\b{X} = \spec \,\mathcal{R}(X)$.
The $\Cl(X)$-grading of $\mathcal{R}(X)$
defines an action of the 
\emph{characteristic quasitorus}
$H := \spec\, \KK[\Cl(X)]$ on $\b{X}$ 
and we recover $X$ as a GIT quotient 
$X = \widehat X \quot H$ of the set of 
semistable points $\widehat X \subseteq \b{X}$ 
associated with any ample class 
$[D] \in \Cl(X) = \Chi(H)$ of~$X$;
see~\cite[Sec~3.2.1]{ArDeHaLa} for details.

In the study of the automorphism group $\Aut(X)$,
several related groups are important.
The~\emph{$H$-equivariant automorphisms} of $\b{X}$ 
are pairs $(\phi, \tilde \phi)$, where  
$\phi \colon \b X \to \b X$ is an automorphism 
of varieties and $\widetilde \phi \colon H \to H$
is an automorphism of linear algebraic groups 
such that for all $x \in \b{X}$ and $h \in H$
we have
$$
\phi(t\cdot x)  
\ = \ 
\tilde\phi(t) \cdot \phi(x).
$$
Write $\aut_H(\b X)$ for the group of all such 
automorphisms of $\b{X}$ and, analogously, 
$\aut_H(\widehat X)$ for those of $\widehat X$.
Also the group $\Bir_2(X)$ of birational
automorphisms of $X$ defined on an open subset
of $X$ having complement of codimension at
least two plays a role.
All these groups are affine algebraic 
and are related to each other via the following 
commutative diagram where the rows are exact
sequences and the upward inclusions 
are of finite index:
$$
\xymatrix{
1
\ar[r]
&
H
\ar[r]
&
{\Aut_H(\b{X})}
\ar[r]
&
{{\Bir_2(X)}}
\ar[r]
&
1
\\
1
\ar[r]
&
H
\ar[r]
% \ar@{=}[u]
\ar@{}[u]|{\rotatebox{90}{\text{$=$}}}
&
{\Aut_H(\rq{X})}
\ar[r]
% \ar@{^{(}->}[u]
\ar@{}[u]|{\rotatebox{90}{\text{$\subseteq$}}}
&
{\Aut(X)}
\ar[r]
% \ar@{^{(}->}[u]
\ar@{}[u]|{\rotatebox{90}{\text{$\subseteq$}}}
&
1,
}
$$
see Theorem~\cite[Thm.~2.1]{ArHaHeLi}.
Moreover, the subgroup $\caut_H(\b X)\subseteq\aut_H(\b X)$,
consisting of all pairs of the form $(\phi,\id)\in\aut_H(\b X)$,
occurs according to~\cite[Cor.~2.8]{ArHaHeLi} 
in the exact sequence 
$$
\xymatrix{
1
\ar[r]
&
H
\ar[r]
&
{\CAut_H(\b{X})^0}
\ar[r]
&
{\Aut(X)^0}
\ar[r]
&
1.
}
$$

\begin{remark}
\label{rem:X2R}
The link to the results of the preceding sections 
is the following. 
Set $R := \mathcal{R}(X)$ and $K := \Cl(X)$.
Then $\caut_H(\b X)$ is isomorphic to $\CAut_K(R)$ 
and $\aut_H(\b X)$ is isomorphic to $\Aut_K(R)$, 
see~\cite[Cor.~2.3]{ArHaHeLi}.
\end{remark}

A very first observation is a bound on the dimension 
of the automorphism group of a Mori dream space in
terms of its Cox ring and the rank its divisor class 
group.

\begin{proposition}
\label{prop:autbound}
Let $X$ be a Mori dream space with Cox ring $\mathcal{R}(X)$ 
and let $\Omega_X \subseteq \Cl(X)$ denote the set of 
generator degrees. Then we have
$$ 
\dim(\Aut(X))
\ \le \ 
\left(
\sum_{w \in \Omega_X} \dim(\mathcal{R}(X)_{w})^2 
\right)
-
\dim(\Cl(X)_\QQ).
$$ 
\end{proposition}

\begin{proof}
The divisor class group $\Cl(X)$ 
is the character group of the 
characteristic quasitorus $H$ and 
thus the dimension of $H$ equals 
the dimension of the rational vector 
space $\Cl(X)_\QQ$.
The assertion is then an immediate consequence 
the above sequences, Remark~\ref{rem:X2R} 
and Corollary~\ref{cor:autdimbound}.  
\end{proof}

\begin{remark}
As the example of the projective spaces $\PP^n$
shows, the bound of Proposition~\ref{prop:autbound}
is sharp.   
\end{remark}

Another application relates the automorphism group 
of $X$ to that of a certain ambient toric variety.
The choice of a minimal system of $K$-prime 
generators of the Cox ring $R$ gives rise  
to a minimal presentation $\KT{r} \to R$.
Fixing an ample class $[D] \in \Cl(X)$ of 
$X$ leads to a commutative diagram
$$ 
\xymatrix{
{\b X}
\ar[r]
&
{\b Z}
\\
{\widehat X}
\ar[r]
\ar@{}[u]|{\rotatebox{90}{\text{$\subseteq$}}}
\ar[d]_{\quot H}
&
{\widehat Z}
\ar@{}[u]|{\rotatebox{90}{\text{$\subseteq$}}}
\ar[d]^{\quot H}
\\
X
\ar[r]
&
Z
}
$$
where $\b{Z} = \KK^r$, the subsets 
$\widehat X \subseteq \b X$ and 
$\widehat Z \subseteq \b Z$ 
are the sets of semistable points
defined by $[D]$
and the horizontal arrows are closed 
embeddings.
The quotient variety $Z = \widehat Z \quot H$ 
is a projective toric variety.
We refer to~\cite[Sec.~3.2.5]{ArDeHaLa} 
for the details of this construction.

\begin{theorem}
\label{thm:bir2}
Consider the closed embedding $X \subseteq Z$ 
into the toric variety $Z$ given above.
\begin{enumerate}
\item
There is an open subset $U \subseteq Z$ with 
$U \cap X \ne \emptyset$ such that all maps 
of $\Bir_2(Z)$ induce automorphisms of $U$.
\item
According to~(i), the stabilizer $\Stab_X(\Bir_2(Z))$ 
is well-defined and, with the subgroup $\Cent_X(\Bir_2(Z))$
of elements defining the identity on $X$, we have
$$ 
\qquad
\Bir_2(X) 
\ \cong \ 
\Stab_X(\Bir_2(Z)) / \Cent_X(\Bir_2(Z)).
$$
\item
With the stabilizer $\Stab_X(\Aut(Z))$ of $X$ 
and the subgroup $\Cent_X(\Aut(Z))$ of all 
elements leaving all points of $X$ fixed, 
we have
$$ 
\qquad
\Aut(X)^0 
\ \cong \ 
(\Stab_X(\Aut(Z)) / \Cent_X(\Aut(Z)))^0.
$$
\item
Let $w_i := \deg(T_i) \in \Cl(X)$ be the
degrees of the generators $T_i \in S = \KT{r}$.
If $S_{w_1} \cup \ldots \cup S_{w_r}$ contains no
relations of $\mathcal{R}(X)$, then we have   
$$ 
\qquad
\Bir_2(X) 
\ \cong \ 
\Stab_X(\Bir_2(Z)),
\qquad
\Aut(X)^0 
\ \cong \ 
\Stab_X(\Aut(Z))^0.
$$ 
\end{enumerate}
\end{theorem}

\begin{proof}
For~(i) recall from the proof of~\cite[Thm.~2.1]{ArHaHeLi}
that the image $U \subseteq Z$ under $\widehat Z \to Z$ of 
the intersection over all possible nonempty sets of 
semistable points of the $H$-action on $Z$ is a nonempty 
open set, where all elements of $\Bir_2(Z)$ define 
automorphisms. By this construction, we have 
$U \cap X \ne \emptyset$. 

Assertion~(ii) is proved by relating the groups 
$\Bir_2(Z)$ and $\Bir_2(X)$ to each other via 
the following commutative diagram, where all 
vertical and horizontal sequences are exact: 
$$ 
\xymatrix{
1 
%\ar[rr]
\ar[d]
&
1 
%\ar[rr]
\ar[d]
&
\ar[d]
1
&
\\
H 
\ar@{}[r]|=
\ar[d]
&
H 
\ar@{}[r]|=
\ar[d]
&
H
\ar[d]
&
\\
{\Aut_H(\b{Z}})
\ar@{}[r]|{\supseteq\qquad}
\ar[d]
&
{\Stab_{\b{X}}(\Aut_H(\b Z))}
\ar[r]
\ar[d]
&
{\Aut_H(\b X)}
\ar[r]
\ar[d]
&
1
\\
{\Bir_2}(Z)
\ar@{}[r]|{\supseteq\qquad}
\ar[d]
&
{\Stab_X(\Bir_2(Z))}
\ar[r]
\ar[d]
&
{{\Bir_2(X)}}
\ar[r]
\ar[d]
&
1
\\
1
&
1
&
1
&
}
$$
To obtain the diagram, we combined 
Proposition~\ref{prop:stab} via Remark~\ref{rem:X2R}
with the upper horizontal sequence 
of~\cite[Thm.~2.1]{ArHaHeLi} stated before.

Assertion~(iii) is proved in the same manner as~(ii), 
one just uses the vertical sequence involving $\CAut_K(R)$ 
of Proposition~\ref{prop:stab} 
and the sequence of~\cite[Cor.~2.8]{ArHaHeLi} given 
above.
Assertion~(iv) is then a direct application of 
Proposition~\ref{prop:stab}~(ii).
\end{proof}

It seems that details on the full automorphism 
group $\Aut(X)$ are not directly visible from 
the picture developed above.
Even in the more accessible special case 
of $X$ having a torus action of complexity 
one, where for example the root subgroups 
can be determined explicitly~\cite{ArHaHeLi},
we do not know how to recognize instantaneously 
the dimension or the number of its connected 
components.

However, we may study $\Aut(X)$ by means 
of the algorithms developed in  the 
preceding sections. 
As in Section~\ref{sec:symalg}, we fix 
a minimal presentation $R=S/I$
with $S=\KT{r}$ for the $K$-graded 
Cox ring $R$ of $X$.
As mentioned before, $\widehat X \subseteq \b X$ 
is the set of semistable points of 
an ample class $w = [D]$ in $K = \Cl(X)$.
In fact, the ample cone of $X$ is the relative 
interior of the \emph{GIT-cone} 
$\lambda \subseteq K_\QQ$ 
determined by $w$, see, e.g.,~\cite{ArDeHaLa, BeHa}.

\begin{algorithm}[Compute $\aut_H(\widehat X)$]
\label{algo:autwidehatX}
\emph{Input:} the $K$-graded Cox ring $R=S/I$ and an
ample class $w \in K$ of a Mori dream space~$X$.
\begin{itemize}
\item 
Compute the GIT-cone $\lambda \subseteq K_\QQ$ 
of $w$, e.g., using~\cite{Ke}.
\item
Compute the subgroup 
$\Sigma \subseteq \Aut(\Orig_S)$ 
stabilizing $\lambda \subseteq K_\QQ$.
\item 
Run Algorithm~\ref{algo:quotrep} 
with input $S$, $I$ and $\Sigma$.
\end{itemize}
\emph{Output:} 
the ideal $J''\subseteq \KK[T_{ij};\,1\leq i,j\leq k]$
computed by Algorithm~\ref{algo:quotrep}.
It describes $\aut_H(\widehat X)\subseteq \GL(k)$. 
\end{algorithm}

\begin{proof}
The subgroup $\aut_H(\widehat X) \subseteq 
\aut_H(\b X)$ is the stabilizer of the subset
$\widehat X\subseteq \b X$. 
By definition of the set $\widehat X$ 
of semistable points associated with 
the ample class $w$,
the graded algebra automorphisms 
defining elements of $\aut_H(\widehat X)$ 
are precisely the $(\varphi,\psi) \in \Aut_K(R)$
such that $\psi$ stabilizes the GIT-cone $\lambda$.
\end{proof}

Passing from $\aut_H(\widehat X)$ to 
$\Aut(X)$ involves an invariant ring 
computation, which in our setting 
amounts to compute a degree zero 
\emph{Veronese subalgebra}. 
Recall that for a $K$-graded $\KK$-algebra
$R$ and a subgroup $K' \subseteq K$, 
the associated Veronese subalgebra is
$$ 
R(K') 
\ := \ 
\bigoplus_{w \in K'} R_w
\ \subseteq \ 
\bigoplus_{w \in K} R_w
\ = \
R.
$$
For any graded presentation $R = S/I$ with 
$S = \KT{r}$, the \emph{degree map} is the 
linear homomorphism $\delta \colon \ZZ^r \to K$ 
with $\delta(e_i) = \deg(T_i) \in K$.

\begin{algorithm}[Computing Veronese subalgebras]
\label{algo:veronese}
\emph{Input:} 
a $K$-graded $\KK$-algebra $R=S/I$ where $S=\KT{r}$ 
and a subgroup $K' \subseteq K$.
\begin{itemize}
\item
Compute generators $\mu_1,\ldots,\mu_s$ for the 
monoid $\delta^{-1}(K') \cap \ZZ_{\ge 0}^r$.
\item 
Compute the preimage $\Psi^{-1}(I)$ under 
$\Psi \colon \KK[Y_1,\ldots,Y_s] \to S$, $Y_j \mapsto T^{\mu_j}$.
 \end{itemize}
\emph{Output: }
the ideal $\Psi^{-1}(I) \subseteq \KK[Y_1,\ldots,Y_m]$.
Then $R(K_0)$ is isomorphic to the $K'$-graded $\KK$-algebra $\KK[Y_1,\ldots,Y_m]/\Psi^{-1}(I)$. 
\end{algorithm}

\begin{remark}
The first step of Algorithm~\ref{algo:veronese} can be carried 
out by a Hilbert basis computation for a pointed, convex, rational
cone.
\end{remark}

In order to apply Algorithm~\ref{algo:veronese},
we have to identify the characteristic quasitorus 
$H = \Spec \, \KK[K]$ of $X$ inside 
$\aut_H(\widehat X)$ in terms of the description 
as a subgroup of $\GL(k)$ as provided by 
Algorithm~\ref{algo:autwidehatX}.
The key observation for this is the following.

\begin{remark}
\label{rem:grading}
Consider the monomial basis 
$\BBB=(b_1,\ldots,b_n)$ 
of the vector subspace $V \subseteq R$ 
as in Construction~\ref{con:repres}.
Setting $\deg(T_{ij}) := u_j := \deg(b_j) \in K$,
we define a $K$-grading on 
$$
\mathcal{O}(\GL(n))
\ = \ 
\KK[T_{ij}; \; 1 \le i,j \le n]_{\det}.
$$
Observe that $\det$ is $K$-homogeneous.
The action of $h \in H$ on $A \in \GL(n)$
is then given by multiplying from the right 
to $A$ the diagonal matrix 
$$
\diag(\chi^{u_1}(h), \ldots, \chi^{u_n}(h)).
$$ 
The ideals provided by Algorithms~\ref{algo:autks}
and~\ref{algo:stab} are $K$-homogeneous and,
finally, the $K$-grading of $\mathcal{O}(\GL(n))$ 
induces a $K$-grading of $\mathcal{O}(\GL(k))$,
where $\GL(k)$ is the ambient group of 
$\Aut_{K,\Sigma}(R)$ provided by 
Algorithm~\ref{algo:quotrep},
which in turn represents $H$ as a subgroup 
of $\aut_H(\widehat X)$.
\end{remark}

\begin{algorithm}[Compute the Hopf algebra $\OO(\Aut(X))$]
\label{algo:autx}
\emph{Input:} 
the $K$-graded Cox ring $R=S/I$
and an ample class $w\in K$
of a Mori dream space~$X$.
\begin{itemize}
\item
Compute with Algorithm~\ref{algo:autwidehatX}
the ideal 
$J'' \subseteq S':=\KK[T_{ij};\,1\leq i,j\leq k]$ 
of $\Aut_H(\widehat X)\subseteq \GL(k)$.
\item 
Install the $K$-grading from Remark~\ref{rem:grading} 
on~$S'$ and compute with Algorithm~\ref{algo:veronese} 
a presentation $(S'/J'')(0) = \KK[Y_1,\ldots,Y_m]/ J$.
\end{itemize}
\emph{Output:}
$\KK[Y_1,\ldots,Y_m]/J$. This is the Hopf algebra 
$\mathcal{O}(\Aut(X))$ of $\Aut(X)$.
\end{algorithm}

\begin{proof}
From the previous discussion, we have 
$\Aut(X)\cong \Aut_H(\widehat X)/H$.
To mod out the subgroup $H$, one passes to the 
degree zero Veronese subalgebra of the  
$K$-grading of $\mathcal{O}(\Aut_H(\widehat X))$
established in Remark~\ref{rem:grading}.
\end{proof}

\begin{remark}
In Algorithm~\ref{algo:autx}, the Hopf algebra structure 
on $\mathcal{O}(\Aut(X))$ is inherited from that of 
$\GL(k)$ via inclusion and restriction
$$
\mathcal{O}(\GL(k))
\ \to\ 
\mathcal{O}\left(\Aut_H(\widehat X)\right)
\ \supseteq\ 
\mathcal{O}(\Aut(X)).
$$
\end{remark}

\begin{example}[$A_32A_1$-singular Gorenstein log del Pezzo $\KK^*$-surface II]
\label{ex:running3}
We continue Example~\ref{ex:running1}.
Since $X$ is a surface, 
there is exactly one GIT-cone $\lambda(w)$ within $\Mov(X)$;
we can choose $w := (2,1)\in K_\QQ$ and obtain
\begin{center}
\begin{minipage}{5cm}
\begin{eqnarray*}
\Aut_H\bigl(\widehat X\bigr)
&=&
\Aut_H(\b X)\\
&\cong&
\Aut_K(R)\\
&=&
% (\ZZ/2\ZZ \times \ZZ/2\ZZ) \ltimes (\KK^*)^3.
\ZZ/2\ZZ \ltimes (\ZZ/2\ZZ \times (\KK^*)^3).
\end{eqnarray*}
\end{minipage}
\qquad 
\begin{minipage}{4cm}
\begin{center}
\footnotesize
\begin{tikzpicture}[scale=.75]

\fill[color=black!45] (0,0) -- (1,1) -- (1,-1) -- cycle;
\fill[color=black!25!white] (0,0) -- (1,0) -- (1,1) -- cycle;

\draw[->, dashed] (-.5,0) -- (1.4,0);
\draw[->, dashed] (0,-1.13) -- (0,1.13);

\draw[thick] (0,0) -- (1,1);
\draw[thick] (0,0) -- (1,0);
\draw[thick] (0,0) -- (1,-1);
\draw[line width=1pt, color=black] (0,0) -- (1.95,.95) node[anchor=north]{$w$};

\fill[color=black] (1,1) circle (1.6pt) node[anchor=west]{$w_1, w_5$};
\fill[color=black] (1,0) circle (1.6pt) node[anchor=north west]{$w_3, w_4$};
\fill[color=black] (1,-1) circle (1.6pt) node[anchor=west]{$w_2$};

\draw[color=black!80!white,decorate,decoration={brace,amplitude=7pt},rotate=0] (2.4,1) -- (2.4,0);
\draw[color=black] (2.5,.5) node[anchor=west]{$\ \  \lambda(w)$};

\end{tikzpicture}
\end{center}

\end{minipage}
\end{center}
We can then compute $\Aut(X)$ with Algorithm~\ref{algo:autx}.
As an affine variety, it is given by
\begingroup
\footnotesize
\begin{gather*}
\Aut(X)\ \cong\ 
V(\KK^{16};\ 
T_{16},
T_{14}+T_{15}-1,
T_{11},
T_{10},
T_{7},
T_{5},
T_{4}-T_{6}+T_{12},
T_{3}-T_{8}-T_{9},
\\\qquad\qquad\qquad
T_{2}-T_{9},
T_{1}-T_{8},
T_{15}^2-T_{15},
T_{12}T_{15}-T_{12},
T_{9}T_{15},
T_{8}T_{15}-T_{8},
\\\qquad\qquad\qquad
T_{6}T_{15}-T_{12},
T_{8}T_{12}-T_{15},
T_{6}T_{9}+T_{15}-1,
T_{9}^2T_{13}-T_{6}^2+T_{12}^2,
\\\qquad\qquad\qquad
T_{8}^2T_{13}-T_{12}^2
),
\end{gather*}
\endgroup
As a group, one verifies that 
$\Aut(X)$ is isomorphic to $\ZZ/2\ZZ \ltimes \KK^*$.
The four components of $\Aut_H(\widehat X)$ are mapped 
onto the two components of $\Aut(X)$, which in turn 
are given by
\begin{center}
\tiny
\begin{minipage}{4.5cm}
\begin{gather*}
V(\KK^{16};\,
T_{16},
T_{15}-1,
T_{14},
T_{11},
\\
T_{10},
T_{9},
T_{7},
T_{6}-T_{12},
T_{5},
\\
T_{4},
T_{3}-T_{8},
T_{2},
T_{1}-T_{8},
\\
T_{8}T_{12}-1,
T_{8}^2T_{13}-T_{12}^2
).
\end{gather*}
\end{minipage}
\qquad 
\begin{minipage}{4.5cm}
\begin{gather*}
V(\KK^{16};\,
T_{16},
T_{15},
T_{14}-1,
T_{12},
T_{11},
\\
T_{10},
T_{8},
T_{7},
T_{5},
T_{4}-T_{6},
\\
T_{3}-T_{9},
T_{2}-T_{9},
T_{1},
\\
T_{6}T_{9}-1,
T_{9}^2T_{13}-T_{6}^2
)
\end{gather*}
\end{minipage}
\end{center}
\end{example}

%%%%%%%%%%%%%%%%%%%%%%%%%%%
\section{Applications}
\label{sec:cubic}

We discuss three applications.
In the first one, we consider the blow-up 
of a projective space $\PP^3$ in a special 
configuration of six points; recall that 
the two-dimensional analogue leads to 
generalized del Pezzo surfaces.
We check if our variety admits an effective 
two-torus action, as a naive glimpse at the 
defining equations of its Cox ring could 
suggest.

\begin{example}
\label{ex:detectcplx1}
Consider the blow up $X\to \PP^3$
along the points $[1, 0, 0, 1]$, $[0, 1, 1, 0]$ and the four standard
toric fixed points.
According to~\cite[Thm.~7.1(iv)]{HaKeLa}, 
the Cox ring $R$ of $X$ is given by

 \begin{center}
 \begin{minipage}{3cm}
  \begingroup
 \footnotesize
  \begin{gather*}
  R\ =\ \KT{12}/I\\
  \text{ where $I$ is generated by}\\
  T_{3}T_{8}-T_{5}T_{12}-T_{2}T_{9},\\
  T_{4}T_{7}-T_{6}T_{11}-T_{1}T_{10}.
\end{gather*}
\endgroup
 \end{minipage}
\qquad\
\begin{minipage}{6cm}
  \begin{gather*}
 \left[
 \mbox{\tiny $
  \begin{array}{rrrrrrrrrrrr}
    1 & 1 & 1 & 1 & 1 & 1 & 0 & 0 & 0 & 0 & 0 & 0\\
  0 & -1 & -1 & -1 & -1 & 0 & 1 & 0 & 0 & 0 & 0 & 0\\
  0 & 1 & 0 & 0 & 1 & 0 & 0 & 1 & 0 & 0 & 0 & 0\\
  0 & 0 & 1 & 0 & 1 & 0 & 0 & 0 & 1 & 0 & 0 & 0\\
  -1 & -1 & -1 & 0 & -1 & 0 & 0 & 0 & 0 & 1 & 0 & 0\\
  0 & -1 & -1 & 0 & -1 & -1 & 0 & 0 & 0 & 0 & 1 & 0\\
  0 & 1 & 1 & 0 & 0 & 0 & 0 & 0 & 0 & 0 & 0 & 1
  \end{array}
  $}
  \right]
 \end{gather*}
 \end{minipage}
 \end{center}
and the columns of the matrix are the degrees
$\deg(T_i)\in K:=\ZZ^7$ of the variables $T_i$.
Using Algorithm~\ref{algo:stab}, we obtain that $\Aut_K(R)$ 
is of dimension $8$.
Thus $\Aut(X)$ is one-dimensional and
cannot contain a two-dimensional torus. 
\end{example}

The second application concerns
singular cubic surfaces $X\subseteq \PP^3$ 
with at most ADE-singularities.
There are~$20$ singularity types of such surfaces; 
in some types, there occur infinite families
due to the existence of parameters in the
equation describing the surface in~$\PP^3$; 
we refer to~\cite[Sec.~8, 9]{Do} for details.
The Cox rings of singular cubic surfaces with  
at most ADE singularities have been determined 
in~\cite{DeHaHeKeLa, De}.
So far, the automorphism groups have been 
computed in the cases without parameters~\cite{Sa}.

\begin{theorem}
\label{thm:cubic}
Let $X \subseteq \PP^3$ be a
singular cubic surface
with at most ADE singularities
and parameters in its defining
equations.
Depending on the ADE singularity
type $S(X)$,
the automorphism group $\aut(X)$
for a general choice of parameters
is the following.

\begingroup
\footnotesize
\begin{longtable}{cccccccc}
\myhline
$S(X)$ & $A_3$ & $A_2A_1$ & $2A_2$  & $3A_1$ & $A_2$ & $2A_1$ & $A_1$  \\
\hline %%%%%%%%%%%%%%%%%%%%%%%%%%%%%%%%%%%%%%%%%%%
$\Aut(X)$ & $\ZZ/2\ZZ$ & $\{1\}$ & $\KK^*\rtimes \ZZ/2\ZZ$ & $S_3$ &
$\{1\}$ & $\ZZ/2\ZZ$ & $\{1\}$\\
\myhline
\end{longtable}
\endgroup
\end{theorem}

\begin{remark}
\label{rem:Hideal}
Consider the output of Algorithm~\ref{algo:quotrep}.
The ideal describing $H = \Spec\,\KK[K]$ inside 
$\GL(k)$ is obtained as follows.
Let $Q$ denote the degree matrix of the 
induced $K$-grading of $\mathcal{O}(\GL(k))$ 
from Remark~\ref{rem:grading} and 
compute a Gale dual for $Q$, i.e., a matrix 
$P$ fitting into
\[
\xymatrix{
0
&
\ar[l]
K
&
\ar[l]^Q
\ZZ^{k^2}
&
\ar[l]^{P^*}
\ker(Q)
&
\ar[l]
0
}.
\]
Next compute the lattice ideal 
$I(P)\subseteq \KK[T_{ij};\ 1\leq i,j\leq k]$ associated 
with $P$ through $E_k\in \GL(k)$, see e.g.~\cite[Rem.~5.5]{HaKeLa}.
Then $I(P)$ describes $H \subseteq \CAut_K(R)$ inside  
$\GL(k)$.
\end{remark}

\begin{proof}[Proof of Theorem~\ref{thm:cubic}]
The families of Cox rings of the resolutions of 
cubic surfaces have been determined in~\cite{DeHaHeKeLa}. 
The Cox rings of the cubic surfaces are obtained 
by contracting all $(-2)$-curves according 
to~\cite[Prop.~2.1]{HaKeLa}, respectively.
The cases involving parameters have the singularity 
types listed in the table. 
Write $R:=\Cox(X)$ and $K:=\Cl(X)$ for the 
respective cubic surface~$X$.

\emph{Cases $A_3$, $A_2A_1$, $3A_1$, $A_2$, $2A_1$ and $A_1$:}
In each case, we use Algorithm~\ref{algo:quotrep} to compute 
the ideal describing $\CAut_K(R) \subseteq \GL(k)$.
Remark~\ref{rem:Hideal} delivers the ideal describing 
the inclusion $H \subseteq \GL(k)$ of the 
characteristic quasitorus $H$ of $X$. 
Comparing ideals, we verify that $H = \CAut_K(R)$ holds.
Therefore, $\Aut(X)$ is isomorphic to $\Gamma = \Aut_K(R)/H$.
The group $\Gamma$ is obtained using 
Remark~\ref{rem:cautgamma}.

\textit{Case $2A_2$: }
The characteristic quasitorus $H$ is isomorphic
to $(\KK^*)^3$ and we have $R=S/I$ with $S=\KT{7}$ and
an ideal~$I\subseteq S$.
In the setting of Construction~\ref{con:repres},
we have $\BBB = (T_1,\ldots,T_7)$ and $\Omega_S=\{w_1,\ldots,w_7\}$
with $w_i:=\deg(T_i)$.
An application of Algorithm~\ref{algo:autks} shows that
$\Aut_K(S)$ consists of the matrices $A,AB\in \GL(7)$,
where $A$ is a $\Omega_S$-diagonal matrix and $B$ 
a $\Aut(\Omega_S)$-permuting matrix as follows:
\begin{gather}
A\,=\,\diag(*,\ldots,*)\ \in\ \GL(7),\qquad
B\,=\,
 \left[
\mbox{\tiny $
    \begin{array}{rrrrrrr}
    0 & 0 & 0 & 1 & 0 & 0 & 0 \\
    0 & 0 & 1 & 0 & 0 & 0 & 0 \\
    0 & 1 & 0 & 0 & 0 & 0 & 0 \\
    1 & 0 & 0 & 0 & 0 & 0 & 0 \\
    0 & 0 & 0 & 0 & 1 & 0 & 0 \\
    0 & 0 & 0 & 0 & 0 & 0 & 1 \\
    0 & 0 & 0 & 0 & 0 & 1 & 0
    \end{array}
    $}
\right]\ \in\ \GL(7).
\label{eq:cubic}
\end{gather}
Using Algorithm~\ref{algo:stab}, we compute an ideal 
$J\subseteq \KK[T_{i,j};\,1\leq i\leq 7]$ 
describing $\Aut_K(R)\subseteq \GL(7)$; it is given by
\begin{gather*}
J \,=\, J'\cdot B^*J'\quad\text{with}\\
J'\,=\,
\left\<
\mbox{\tiny $
\begin{array}{ll}
% T_{1,1}T_{2,2}T_{3,3}T_{4,4}T_{5,5}T_{6,6}T_{7,7}Z-1, & 
T_{2,2}T_{3,3}-T_{1,1}T_{4,4},&
-aT_{5,5}^2+aT_{6,6}T_{7,7},\\
(a-1)T_{1,1}T_{4,4}+(-a+1)T_{5,5}^2, & 
T_{2,2}T_{3,3}-T_{6,6}T_{7,7},\\
\end{array}
$}
\right\>
+ \<T_{i,j};\, i\ne j\>,
\end{gather*}
where $a\in \KK^*\setminus \{1\}$ is a parameter coming 
from the ideal of $X\subseteq \PP^3$.
Since $B$ is $\Aut(\Orig_S)$-permuting, there is 
$\sigma_B\in \Aut(\Orig_S)$ such that $\deg(T_{j,j}) = \sigma_B(\deg(T_{i,i}))$.
By the previous decomposition of $\Aut_K(S)$, 
we obtain a homomorphism
$$
\alpha\colon \Aut_K(R) \ \to \ \Gamma\,\cong\,\ZZ/2\ZZ,
\qquad
B'A'\,\mapsto\,
\sigma_{B'}.
$$
Moreover, we have a homomorphism 
$\beta\colon \KK^*\times H \to \Aut_K(R)$
that maps a tuple $(v,s,t,u)$ to the product 
$A_v\cdot B_s\cdot C_t\cdot D_u$ where
\begin{gather*}
A_v:=\diag(v^{-1},1,1,v,1,1,1),\qquad
B_s:=\diag(s^{-1},s,s^{-1},s,1,1,1),\\
C_t:=\diag(t^{-2},t,t^{-3},1,t^{-1},t^{-2},1),\qquad
D_u:=\diag(1,u,u^{-1},1,1,u^{-1},u).
\end{gather*}
One directly checks that $\beta$ is well-defined.
Comparing entries in $A_v\cdot B_s\cdot C_t\cdot D_u$, 
we see that $\beta$ is injective.
We thus have a split exact sequence
$$
\xymatrix{
1
\ar[r]
&
\KK^*\times H
\ar[r]^{\beta}
&
\Aut_K(R)
\ar[r]^{\alpha}
&
\Gamma
\ar[r]
%\ar@{-->}@/^1pc/[l]^\gamma
\ar@/^1pc/[l]^\gamma
&
1
},
$$
where we define the homomorphism 
$\gamma\colon \Gamma\to \Aut_K(R)$  
by mapping $\id$ to the $7\times 7$-unit matrix and 
$\sigma_B$ with $B$ as in~\eqref{eq:cubic}
to the element $BA'\in \Aut_K(R)$,
where $A':=\diag(1,\ldots,1)$.
Thus, we have a semidirect product representation for 
$\Aut_K(R)$ which in turn descends to 
\[
\Aut(X)\ \cong\ \Aut_K(R)/H\ \cong\  \KK^*\rtimes \ZZ/2\ZZ.
\qedhere
\]
\end{proof}

\begin{remark}
Note that passing to special parameters for the surfaces 
of Theorem~\ref{thm:cubic} may lead to larger 
automorphism groups.
Whereas Algorithms~\ref{algo:stab} and~\ref{algo:autx}  
return equations including parameters, the description of 
the group needs to be done anew.
As an example of this effect, consider the $\ZZ^7$-graded 
rings 
\begin{gather*}
R_\lambda
\ := \
\KT{10}
/
\<T_{2}T_{5}^2T_{8}+T_{3}T_{6}^2T_{9}+T_{4}T_{7}^2T_{10}
-\lambda T_{1}T_{2}T_{3}T_{4}T_{5}T_{6}T_{7}\>,
\\
\left[
\mbox{\tiny$\begin{array}{rrrrrrrrrr}
  1 & 0 & 0 & 0 & 0 & 0 & 0 & 1 & 1 & 1 \\
  -1 & 1 & 0 & 0 & 0 & 0 & 0 & -1 & 0 & 0 \\
  -1 & 0 & 1 & 0 & 0 & 0 & 0 & 0 & -1 & 0 \\
  -1 & 0 & 0 & 1 & 0 & 0 & 0 & 0 & 0 & -1 \\
  0 & -1 & 0 & 0 & 1 & 0 & 0 & -1 & 0 & 0 \\
  0 & 0 & -1 & 0 & 0 & 1 & 0 & 0 & -1 & 0 \\
  0 & 0 & 0 & -1 & 0 & 0 & 1 & 0 & 0 & -1
\end{array}
$}
\right]
\end{gather*}
where $\lambda\in \{0,1\}$ and the columns of the matrix 
are the degrees $\deg(T_i)\in \ZZ^7$.
The rings $R_\lambda$ are the Cox rings $\Cox(X_\lambda)$ of the minimal resolutions $X_\lambda\to X_\lambda'$ of
cubic surfaces $X_\lambda'$ with singularity type 
$D_4$, see~\cite[case $D_4$, p.~28]{De}. 
One has $X_\lambda \cong X_1$ for $\lambda \ne 0$ 
and from~\cite[Thm.~2]{Sa}, we infer
$$
\Aut(X_1) \ \cong \ S_3,
\qquad
\Aut(X_0) \ \cong \ \KK^*\rtimes S_3.
$$
Observe that Proposition~\ref{prop:autbound} 
gives the bound $\dim(\Aut(X_\lambda)) \le 3$,
which is not sharp in the present case.
\end{remark}

Our third application is in computer algebra:
Algorithm~\ref{algo:stab} is a linear algebra-based, Gr\"obner
basis-free way to detect symmetries of graded ideals. These symmetries
can then be used to speed up computations on ideals.

\begin{example}\label{ex:groebnerfan}
The affine cone over the Grassmannian $G(2,5)$ has the coordinate ring
$R=S/I$ where $S=\KT{10}$ and $I$ is generated by the Pl\"ucker relations
 \begingroup
 \footnotesize
  \begin{gather*}
 T_{5}T_{10}-T_{6}T_{9}+T_{7}T_{8},\qquad
  T_{1}T_{9}-T_{2}T_{7}+T_{4}T_{5},\qquad
  T_{1}T_{8}-T_{2}T_{6}+T_{3}T_{5},\\
  T_{1}T_{10}-T_{3}T_{7}+T_{4}T_{6},\qquad
  T_{2}T_{10}-T_{3}T_{9}+T_{4}T_{8}.
\end{gather*}
\endgroup
We use that $R$ is effectively and pointedly graded by $\ZZ^5$:
the degrees of the variables $T_i$ are the columns of
 \begin{gather*}
 \left[
 \mbox{\tiny $
  \begin{array}{rrrrrrrrrr}
  1 & 1 & 1 & 1 & 0 & 0 & 0 & 0 & 0 & 0 \\
  1 & 0 & 0 & 0 & 1 & 1 & 1 & 0 & 0 & 0 \\
  0 & 1 & 1 & 0 & 0 & 0 & -1 & 1 & 0 & 0 \\
  0 & 1 & 0 & 1 & 0 & -1 & 0 & 0 & 1 & 0 \\
  0 & 0 & 1 & 1 & -1 & 0 & 0 & 0 & 0 & 1
  \end{array}
  $}
  \right].
 \end{gather*}
Applying Algorithm~\ref{algo:stab} to $R$, 
we obtain $\Stab_I(\Aut_K(S))\cong S_5$.
Ten of the $120$ elements are just permutations 
of variables: written as elements of $S_{10}$, 
we have besides the identity
\begingroup \tiny
\begin{gather*}
  (10,9,7,4,8,6,3,5,2,1), \qquad
 (5,6,7,1,8,9,2,10,3,4) \qquad
(10,3,6,8,4,7,9,1,2,5), \\
(8,9,2,5,10,3,6,4,7,1) \qquad
(8,6,3,10,5,2,9,1,7,4), \qquad
(4,3,2,1,10,9,7,8,6,5) \\
(1,7,6,5,4,3,2,10,9,8), \qquad
(5,2,9,8,1,7,6,4,3,10) \qquad
(4,7,9,10,1,2,3,5,6,8).   
\end{gather*}
 \endgroup
Then the computation of the Gr\"obner fan of $I$
with \texttt{gfan}~\cite{gfan} using the \texttt{symmetries}
option and the ten listed symmetries was computed instantly on a 5 year
old machine (core2duo), whereas the case of no given symmetries took
about two seconds.
\end{example}

\end{document}